\documentclass[10 pt]{article}
\usepackage[centertags]{amsmath}
\usepackage{amssymb}
\usepackage{amsthm}
\usepackage{newlfont}
\usepackage{amsfonts}
\usepackage{times}
\usepackage{amssymb,amsthm,latexsym,epsfig,pgf}
\usepackage[centertags]{amsmath}
\usepackage{newlfont}
\usepackage{amsfonts}
\usepackage{graphicx}
\usepackage{color}
\usepackage[colorlinks]{hyperref}
\usepackage{mathptmx}
\theoremstyle{plain}
\newtheorem{thm}{Theorem}
\newtheorem{cor}{Corollary}
\newtheorem{lem}{Lemma}
\newtheorem{prop}{Proposition}
\newtheorem{conjec}{Conjecture}
\newtheorem{defn}{Definition}

\newcommand{\ra}{\rightarrow}

\newcommand{\sub}{\subseteq}
\newcommand{\mo}{$L=x_1x_2...x_n$}
\newcommand{\bks}{\backslash}

\newcommand{\degr}{$d^+_D(v)\leq d^{++}_D(v)$}
\begin{document}

\begin{center}
\textbf{About the Second Neighborhood Problem in Tournaments Missing Disjoint Stars}
\end{center}

\begin{center}
Salman Ghazal   \footnote{Department of Mathematics, Faculty of Sciences I, Lebanese University, Hadath, Beirut, Lebanon.\\
                       E-mail: salmanghazal@hotmail.com}$^,$
                \footnote{Institute Camille Jordan, D\'{e}partement de Math\'{e}matiques, Universit\'{e} Claude Bernard Lyon 1, France.}
\end{center}

\begin{abstract}
Let $D$ be a digraph without digons. Seymour’s second neighborhood
conjecture states that $D$ has a vertex $v$ such that $d^+(v) \leq d^{++}(v)$. Under
some conditions, we prove this conjecture for digraphs missing $n$ disjoint
stars. Weaker conditions are required when $n = 2$ or $3$. In some cases we
exhibit $2$ such vertices.
  
\end{abstract}

\section{Introduction}
In this paper, a digraph $D$ is a pair of two disjoint finite sets $(V, E)$ such that $E\sub V\times V$. $E$ is the arc set and $V$ is the vertex set and they are denoted by $E(D)$ and $V(D)$ respectively. An oriented graph is a digraph without loop and digon (directed cycles of length two).
If $K\subseteq V(D)$ then the induced restriction of $D$ to $K$ is denoted by $D[K]$. As usual, $N^{+}_D(v)$ (resp. $N^{-}_D(v)$)
denotes the (first) out-neighborhood (resp. in-neighborhood) of a vertex $v\in V$. $N^{++}_D(v)$ (resp. $N^{--}_D(v)$) denotes the second
\textit{out-neighborhood} (\textit{in-neighborhood}) of $v$, which is the set of vertices that are at distance 2 from $v$ (resp. to $v$). We also denote
$d^{+}_D(v)=|N^{+}_D(v)|$, $d^{++}_D(v)=|N^{++}_D(v)|$, $d^{-}_D(v)=|N^{-}_D(v)|$ and $d^{--}_D(v)=|N^{--}_D(v)|$.
We omit the subscript if the digraph is clear from the context.
For short, we write $x\rightarrow y$ if the arc $(x,y)\in E$.
A vertex $v\in V(D)$ is called \textit{whole} if it is adjacent to every vertex in $V(D)-\{v\}$. A \textit{sink} $v$ is a vertex with $d^{+}(v)=0$, while a source $v$ is a vertex with $d^{-}(v)=0$.
For $x,y\in V(D)$, we say $xy$ is a \textit{missing edge} of $D$ if neither $(x,y)$ nor $(y,x)$ are in $E(D)$.
The\textit{ missing graph} $G$ of $D$ is
the graph whose edges are the missing edges of $D$ and whose vertices are the non whole vertices of $D$.
In this case, we say that $D$ is \textit{missing} $G$. So, a tournament does not have any missing edge. A star of center $x$ is a graph whose edge set has the form $\{a_ix; i=1,...,k\}$. In this paper, $n$ stars are said to be disjoint if any two of them do not share a common vertex. \\

\par A vertex $v$ of $D$ is said to have the \textit{second neighborhood property} (SNP) if \degr 
. In 1990, Seymour conjectured the following:\\

\begin{conjec}
\textbf{(Seymour's Second Neighborhood Conjecture (SNC))}\cite{dean}
\hspace{5pt}Every oriented graph has a vertex with the SNP.
\end{conjec}

In 1996, Fisher \cite{fisher} solved the SNC for tournaments by using a certain probability distribution
on the vertices. Another proof of Dean's conjecture was established in 2000 by Havet and Thomass\'{e} \cite{m.o.}. Their short proof uses a tool called median orders. Furthermore, they have proved that if a tournament has no sink vertex then there are at least two vertices
with the SNP. In 2007 Fidler and Yuster \cite{fidler} proved, using median orders and dependency digraphs, that SNC holds for digraphs missing a matching, a star or a complete graph. Ghazal proved more general statements in \cite{gs, gg} and proved that the SNC holds for some other classes of digraphs \cite{contrib}.\\

\section{Definitions and  Preliminary Results}

\par
Let $L=v_1v_2...v_n$  be an ordering of the vertices of a digraph $D$.
An arc $e=(v_i,v_j)$ is \textit{forward} with respect to $L$ if $i<j$. Otherwise $e$ is a \textit{backward} arc. The weight of $L$ is $\omega(L)=|\{(v_i,v_j)\in E(D); i<j\}|$. $L$ is called a \textit{median order} of $D$ if $\omega(L)=max\{\omega(L'); L'$ is an ordering of the vertices of $D\}$; that is $L$ maximizes the number of forward arcs.
In fact, the median order $L$ satisfies the \textit{feedback property}: For all $1\leq i\leq j\leq n:$
$$ d^{+}_{D[i,j]}(v_i)  \geq    d^{-}_{D[i,j]}(v_i)  $$
and
$$ d^{-}_{D[i,j]}(v_j)  \geq    d^{+}_{D[i,j]}(v_j)  $$
where $[i,j]:=\{v_i,v_{i+1}, ...,v_j\}$ (See \cite{m.o.}).\\

It is also known that if we reverse the orientation of a backward arc $e=(v_i,v_j)$ of $D$ with respect to $L$, then $L$ is again a weighted median order of the new digraph $D'=D-(v_i,v_j)+(v_j,v_i)$ (See \cite{contrib}).\\

Let $L=v_1v_2...v_n$ be a median order. Among the vertices not in $N^+(v_n)$ two types are distinguished: A vertex $v_j$ is \textit{good} if there is $i\leq j$ such that
$v_n\ra v_i\ra v_j$, otherwise $v_j$ is a \textit{bad vertex}. The set of good vertices of $L$ is denoted by $G_L^D$ \cite{m.o.} ( or $G_L$ if there is no confusion ).
Clearly, $G_L \sub N^{++}(v_n)$. The last vertex $v_n$ is called a feed vertex of $D$.\\

\par We say that a missing edge $x_1y_1$ \textit{loses to} a missing edge $x_2y_2$ if:
$x_1\rightarrow x_2$, $y_2\notin N^{+}(x_1)\cup N^{++}(x_1)$, $y_1\rightarrow y_2$ and $x_2\notin N^{+}(y_1)\cup N^{++}(y_1)$.
The \textit{dependency digraph} $\Delta$ of $D$ is defined as follows: Its vertex set consists of all the missing
edges and $(ab,cd)\in E(\Delta)$ if $ab$ loses to $cd$ \cite{fidler,contrib}. Note that $\Delta$ may contain digons.\\

\begin{defn}\cite{gs}
In a digraph D, a missing edge $ab$ is called a \emph{good missing edge} if:\\
$(i)$   $(\forall v \in V\backslash\{a,b\})[(v\rightarrow a)\Rightarrow(b\in N^{+}(v)\cup N^{++}(v))]$ or\\
$(ii)$ $(\forall v \in V\backslash\{a,b\})[(v\rightarrow b)\Rightarrow(a\in N^{+}(v)\cup N^{++}(v))]$.\\
If $ab$ satisfies $(i)$ we say that $(a,b)$ is a convenient orientation of $ab$.\\
If $ab$ satisfies $(ii)$ we say that $(b,a)$ is a convenient orientation of $ab$.
\end{defn}

We will need the following observation:
\begin{lem} (\cite{fidler}, \cite{contrib})
 Let $D$ be an oriented graph and let $\Delta$ denote its dependency digraph. A missing edge $ab$ is good
if and only if its in-degree in $\Delta$ is zero.
\end{lem}

Let $D$ be a digraph and let $\Delta$ denote its dependency digraph. Let $C$ be a connected component of $\Delta$.
Set $K(C)=$ $\{u\in V(D);$ there is a vertex $v$ of $D$ such that $uv$ is a missing edge and belongs to $C$ $\}$. The \emph{interval
graph of $D$}, denoted by $\mathcal{I}_D$ is defined as follows. Its vertex set consists of the connected components
of $\Delta$ and two vertices $C_1$ and $C_2$ are adjacent if $K(C_1)\cap K(C_2)\neq \phi$. So  $\mathcal{I}_D$ is the
intersection graph of the family $\{K(C);C$ is a connected component of $\Delta$ $\}$. Let $\xi$ be a connected component of $\mathcal{I}_D$.
We set $K(\xi)=\displaystyle\cup _{C\in \xi}K(C)$. Clearly, if $uv$ is a missing edge in $D$ then there is a unique connected component $\xi$ of
$\mathcal{I}_D$ such that $u$ and $v$ belong to $K(\xi)$. For $f\in V(D)$, we set $J(f)=\{f\}$ if $f$ is a whole vertex, otherwise $J(f)=K(\xi)$,
where $\xi$ is the unique connected component of $\mathcal{I}_D$ such that $f\in K(\xi)$. Clearly, if $x\in J(f)$ then $J(f)=J(x)$ and
if $x\notin J(f)$ then $x$ is adjacent to every vertex in $J(f)$.\\

Let $L=x_{1}\cdots x_{n}$ be a  median order of a digraph $D$. For $i<j$, the sets
$[i,j]:=[x_i,x_j]:=\{x_i,x_{i+1},...,x_j\}$ and  $]i,j[=[i,j]\bks \{x_i,x_j\}$
are called \textit{intervals} of $L$. We recall that $K\subseteq V(D)$ is an \textit{interval of $D$} if for every $u,v\in K$ we have:
$N^+(u)\backslash K =N^+(v)\backslash K$ and $N^-(u)\backslash K =N^-(v)\backslash K$. The following shows a relation
between the intervals of $D$ and the intervals of $L$.

\begin{prop}\cite{gg}
Let $\mathcal{I}=\{I_1,...,I_r\}$ be a set of pairwise disjoint intervals of $D$. Then for every median order $L$ of $D$,
there is a weighted median order $L'$ of $D$ such that: $L$ and $L'$ have the same feed vertex and every interval in $\mathcal{I}$
is an interval of $L'$.
\end{prop}

We say that $D$ is \textit{good digraph} if the sets $K(\xi)$'s are intervals of $D$. By the previous proposition, every good digraph has a  median order $L$
such that the $K(\xi)$'s form intervals of $L$. Such an enumeration is called a \textit{good median order} of the good digraph $D$ \cite{gg}.\\

\begin{thm}\label{prince}\cite{gg}
Let $D$ be a good oriented graph and let L be a good  median order of $D$, with feed vertex f.
Then for every $x\in J(f)$, we have $|N^+(x)\backslash J(f)|\leq |G_L\backslash J(f)|$.
So if $x$ has the  SNP in $D[J(f)]$, then it has the  SNP in $D$.
\end{thm}

\begin{cor}(\cite{m.o.})
Let L be a median order of a tournament with feed vertex f. Then $|N^+(f))| \leq |G_L|$.
\end{cor}

Let $L$ be a good median order of a good oriented graph $D$ and let $f$ denote its feed vertex. By theorem \ref{prince}, for every
$x\in J(f)$, $|N^+(x)\backslash J(f)|\leq |G_L\backslash J(f)|$. Let $b_1,\cdots,b_r$ denote the bad vertices of
$L$ not in $J(f)$ and $v_1,\cdots ,v_s$ denote the non bad vertices of $L$ not in $J(f)$, both enumerated in increasing order with respect
to their index in $L$.\\
If $|N^+(x)\backslash J(f)|< |G_L\backslash J(f)|$, we set $Sed(L)=L$. If $|N^+(x)\backslash J(f)|=|G_L\backslash J(f)|$,
we set $sed(L)=b_1\cdots b_rJ(f)v_1\cdots v_s$. This new order is called the \textit{sedimentation of $L$}.
\begin{lem}\label{gsedlem}\cite{gg}
Let $L$ be a good  median order of a good  oriented graph $D$. Then $Sed(L)$ is a good median order of $D$.
\end{lem}

In the rest of this section, $D$ is an oriented graph missing a matching and $\Delta$ denotes its dependency digraph.
We begin by the following lemma:

\begin{lem}\cite{fidler}\label{structdep}
The maximum out-degree of $\Delta$ is one and the maximum in-degree of $\Delta$ is one. Thus $\Delta$
is composed of vertex disjoint directed paths and directed cycles.
\end{lem}

\begin{proof}
 Assume that $a_{1}b_{1}$ loses to $a_{2}b_{2}$ and $a_{1}b_1$ loses to $a'_{2}b'_{2}$, with $a_1\ra a_2$ and $a_1\ra a'_2$. The edge $a'_2b_2$ is not a missing edge
of $D$. If $a'_2\ra b_2$ then $b_1\ra a'_2\ra b_2$, a contradiction. If $b_2\ra a'_2$ then $b_1\ra b_2\ra a'_2$, a contradiction. Thus,
the maximum out-degree of $\Delta$ is one. Similarly, the maximum in-degree is one.
\end{proof}

In the following, $C=a_1b_1,...,a_kb_k$ denotes a directed cycle of $\Delta$, namely
$a_i\rightarrow a_{i+1}$, $b_{i+1}\notin N^{++}(a_i)\cup N^{+}(a_i)$, $b_i\rightarrow b_{i+1}$
and $a_{i+1}\notin N^{++}(b_i)\cup N^{+}(b_i)$, for all $i<k$.

\begin{lem}
(\cite{fidler})\label{howlose}
If k is odd then $a_k\rightarrow a_1$, $b_1\notin N^{++}(a_k)\cup N^+(a_k)$,
$b_k\rightarrow b_1$ and $a_1\notin N^{++}(b_k)\cup N^+(b_k)$. If k is even then $a_k\rightarrow b_1$, $a_1\notin N^{++}(a_k)\cup N^+(a_k)$,
$b_k\rightarrow a_1$ and $b_1\notin N^{++}(b_k)\cup N^+(b_k)$.

\end{lem}

\begin{lem}\cite{fidler}\label{intervals}
$K(C)$ is an interval of $D$.
\end{lem}

\begin{proof}
Let $f\notin K(C)$. Then $f$ is adjacent to every vertex in $K(C)$.
If $a_1\ra f$ then $b_2\ra f$, since otherwise $b_2\in N^{++}(a_{1})\cup N^{+}(a_{1})$ which is a contradiction.
So $N^{+}(a_{1})\bks K(C)\sub N^{+}(b_{2})\bks K(C)$. Applying this to every losing relation of $C$ yields
$N^{+}(a_{1})\bks K(C)\sub N^{+}(b_{2})\bks K(C)\sub N^{+}(a_{3})\bks K(C)...\sub N^{+}(b_{k})\bks K(C)\sub N^{+}(b_{1})\bks K(C)
\sub N^{+}(a_{2})\bks K(C)...\sub N^{+}(a_{k})\bks K(C)\sub N^{+}(a_{1})\bks K(C)$ if $k$ is even. So these inclusion are equalities.
An analogous argument proves the same result for odd cycles.
\end{proof}

\section{Main Results}

\subsection{Removing n stars}
\par We recall that a vertex $x$ in a tournament $T$ is a king if $\{x\}\cup N^+(x)\cup N^{++}(x)=V(T)$. It is well known
that every tournament has a king. However, for every natural number $n\notin \{2,4\}$, there is a tournament $T_n$ on $n$ vertices, such that every vertex is a king for this tournament.\\

A digraph is called non trivial if it has at least one arc.
\begin{prop}\label{goodstars}
Let $D$ be a digraph missing disjoint stars. If the connected components of its dependency digraph are non-trivial strongly connected, then $D$ is a good digraph.
\end{prop}

\begin{proof}

 Let $\xi$ be a connected component of $\Delta$. First, suppose that $K(\xi)=K(C)$ for some directed cycle $C=a_1b_1,a_2b_2,...,a_nb_n$ in $\Delta$, namely $a_i\ra a_{i+1}$ and $b_{i+1}\notin N^+(a_i)\cup N^{++}(a_i)$. If the set of the missing edges $\{a_ib_i; i=1,...,n\}$ forms a matching, then by lemma \ref{intervals}, $K(C)$ is an interval of $D$. \\

 So we will suppose that a center $x$ of a missing star appears twice in the list $a_1,b_1,a_2,b_2,...,a_n,b_n$ and assume without loss of generality that $x=a_1$. Suppose that $n$ is even. Set $K_1=\{a_1,b_2,...,a_{n-1},b_{n}\}$ and $K_2=K(C)\backslash K_1$.  \\

 Suppose that $a_{n}\rightarrow b_1$ and $a_1\notin N^{+}(a_{n})\cup N^{++}(a_{n})$. Then by following the proof of lemma \ref{intervals} we get the desired result. \\

 Suppose $a_{n}\rightarrow a_1$ and $b_1\notin N^{+}(a_{n})\cup N^{++}(a_{n})$. Then by following the proof of lemma \ref{intervals} we get that $K_1$ and $K_2$ are intervals of $D$. Assume, for contradiction that $K_1\cap K_2 =\phi$ and let $i>1$ be the smallest index for which $x$ is incident to $a_ib_i$. Clearly $i>2$. However, $b_3\notin K_1$ and $x=a_1\ra a_2\ra a_3$ implies that $i>3$.
 Suppose that $x=a_i$. Note that $i$ must be odd by definition of $K_1$. Since $b_2\rightarrow a_1=x=a_i$ and $a_3\notin N^{+}(x)\cup N^{++}(x)$ then
 $a_3\rightarrow x$. Similarly $b_4,a_5,...,b_{i-1}$ are in-neighbors of $x$. However, $b_{i-1}$ is an out-neighbor of $a_i=x$, a contradiction. Suppose that $x=b_i$. Similarly, $a_3,b_4,...,a_{i-1}$ are in-neighbors of $x$. However,
 $a_{i-1}$ is an out-neighbor of $b_i=x$, a contradiction. Thus $K_1\cap K_2\neq \phi$. Whence, $K=K_1\cup K_2$ is an interval of $D$. Similar argument is used to prove it when $n$ is odd.\\

 This result can be easily extended to the case when $K(\xi)=K(C)$ and $C$ is a non trivial strongly connected component of $\Delta$, because between any two missing edges $uv$ and $zt$ there is directed path from $uv$ to $zt$
 and a directed path from $zt$ to $uv$. These two directed paths creat many directed cycles that are used to prove the desired result.\\

 This also is extended to the case when $K(\xi)=\displaystyle\cup_{C\in \xi} K(C)$: Let $u$ and $u'$ be two vertices of $K(\xi)$. There are two non trivial strongly connected components of $\Delta$ such that $u\in K(C)$ and $u'\in K(C')$. Since $\xi$ is a connected component of $\mathcal{I}_D$, there is a path $C=C_0C_1...C_n=C'$. For all $i>0$, there is  $u_i\in K(C_{i-1})\cap K(C_i)$, by definition of edges in $\mathcal{I}_D$. Therefore, $N^+(u)\bks K(\xi)=N^+(u_1)\bks K(\xi)=...=N^+(u_i)\bks K(\xi)=...=N^+(u_n)\bks K(\xi)=N^+(u')\bks K(\xi)$ and $N^-(u)\bks K(\xi)=N^-(u_1)\bks K(\xi)=...=N^-(u_i)\bks K(\xi)=...=N^-(u_n)\bks K(\xi)=N^-(u')\bks K(\xi)$.

\end{proof}

\begin{thm}
Let $D$ be a digraph obtained from a tournament by deleting the edges of disjoint stars. Suppose that, in the induced
tournament by the centers of the missing stars, every vertex is a king. If $\delta^-_{\Delta}>0$ then $D$ satisfies SNC.
\end{thm}

\begin{proof}

Orient every missing edge of $D$ towards the center of its star. Let $L$ be a median order of the obtained tournament $T$ and let $f$ be its feed vertex. Then $f$ has the SNP in $T$. We prove that $f$ has the SNP in $D$ as well.\\

\par First, suppose that $f$ is a whole vertex. Then $N^+(f)=N^+_T(f)$. Let $v\in N^{++}_T(f)$. Then there $\exists u\in V(T)=V(D)$ such that $f\ra u\ra v\ra f$ in $T$. Since $f$ is whole, then $(f,u)$ and $(v,f)\in D$. If $(u,v)\in D$ then $v\in N^{++}(f)$. Otherwise, $uv$ is a missing edge and hence, $\exists ab$ that loses to $uv$, say $b\ra v$ and $u\notin N^+(b)\cup N^{++}(b)$. But $fb$ is not a missing edge, since $f$ is whole. Then $(f,b)\in D$, since otherwise, $b\ra f\ra u$ in $D$ which is a contradiction. Therefore, $f\ra b\ra v$ in $D$. Whence, $v\in N^{++}(f)$. So $N^{++}_T(f)\sub N^{++}(f)$. Therefore, $d^+(f)=d^+_T(f)\leq d^{++}_T(f)\leq d^{++}(f)$.\\

\par Now suppose that $f$ is the center of a missing star. Then $N^+(f)=N^+_T(f)$. Let $v\in N^{++}_T(f)$.  Then there $\exists u\in V(T)=V(D)$ such that $f\ra u\ra v\ra f$ in $T$. Then $(f, u) \in D$ while $(f,v)\notin D$. If $(u,v) \in D$ then $v\in N^{++}(f)$. Otherwise, $uv$ is a missing edge and $v$ is the center of a missing star. Then $v\in N^{+}(f)\cup N^{++}(f)$, because $f$ is a king for the centers of the missing stars. Note that $v\notin N^{+}(f)$. So $N^{++}_T(f)\sub N^{++}(f)$. Therefor, $f$ has the SNP in $D$.\\

\par Finally, suppose that $f$ is not whole and not the center of a missing star. Then $\exists x$ a center of a missing star such that $fx$ is a missing edge. We distinguish between two cases.\\

\par In the first case, we suppose that $fx$ does not lose to any missing edge. We reorient $fx$ as $(x,f)$. Since $(f,x)\in T$ is a backward arc with respect to $L$, the again $L$ is a median order of the new tournament $T'$ obtained by reversing the orientation of $fx$. Moreover, $N^+(f)=N^+_{T'}(f)$ and $f$ has the SNP in $T'$.  Let $v\in N^{++}_{T'}(f)$.  Then there $\exists u\in V(T)=V(D)$ such that $f\ra u\ra v\ra f$ in $T'$. Then $(f, u) \in D$ while $(f,v)\notin D$. If $(u,v) in D$ then $v\in N^{++}(f)$. Otherwise $uv$ is a missing edge and $v$ is the center of a missing star.Since $\Delta$ has no source, there is a missing edge that loses to $uv$. Suppose that this edge is of the form $ax$. Then we must have $x\ra v$ and $u\notin N^+(x)\cup N^{++}(x)$, by definition of losing relation and due to the fact that $v\in N^+(x)\cup N^{++}(x)$ ($x$ is a king for the centers of the missing stars). If $v\notin N^{++}(f)$, then $fx$ loses $uv$ which is a contradiction to the supposition of this case. Hence, $v\notin N^{++}(f)$. Now, suppose that the missing edge that loses to $uv$ is of the form $by$ with $x\notin \{b,y\}$. Suppose without loss of generality that $y$ is the center of a missing star containing $by$. Then $y\ra v$ and $u\notin N^+(y)\cup N^{++}(y)$, by definition of losing relation and due to the fact that $v\in N^+(y)\cup N^{++}(y)$ ($y$ is a king for the centers of the missing stars). But $(f,u)\in D$ and $fy$ is not a missing edge, then $(f,y)\in D$. Thus $f\ra y\ra v$. Whence, $v\in N^{+}(f)\cup N^{++}(f)$. So $N^{++}_{T'}(f)\sub N^{++}(f)$. Therefor, $f$ has the SNP in $D$ as well.\\

\par In the second case, we suppose that $fx$ loses to some missing edge $by$. We may assume without loss of generality that $y$ is the center of a missing star containing $by$. Then we must have $x\ra y$ and $b\notin N^{+}(x)\cup N^{++}(x)$. Clearly, $N^+(f)\cup \{y\}=N^+_T(f)$.  We prove that $N^{++}_{T}(f)\sub N^{++}(f)\cup \{y\}$. Let $v\in N^{++}_{T}(f)\bks y$.  Then there $\exists u\in V(T)=V(D)$ such that $f\ra u\ra v\ra f$ in $T$. Suppose that $u=x$. Since $bv$ is not a missing edge, $x=u\ra v$ and $b\notin N^{+}(x)\cup N^{++}(x)$ then we must have $(b,v) \in D$. Whence, $f\ra b\ra v$ in $D$. Therefore $v\in N^{++}(f)$. Now suppose that $u\neq x$. Then $(f, u)\in D$. If $(u,v)\in D$ then $v\in N^{++}(f)$. Otherwise, $uv$ is a missing edge. Hence there is a missing edge $pq$ that loses to $uv$, namely, $q\ra v$ and $u\notin N^{+}(q)\cup N^{++}(q)$. If $q=x$, then we have $f\ra x\ra v \ra f$ in $T$, which is the same as the case when $u=x$. So we may suppose that $q\neq x$. Note that $q$ must be the center of a missing star. So $f,x\notin \{p,q\}$. Thus $fq$ is not a missing edge, $u\notin N^{+}(q)\cup N^{++}(q)$ and $(f,u) \in D$. Then we must have $(f,q) \in D$, since otherwise we get $q\ra f\ra u$ in $D$ which is a contradiction. Thus $f\ra q \ra v $ in $D$. Whence $v\in N^{++}(f)$. So $N^{++}_{T}(f)\sub N^{++}(f)\cup \{y\}$. Therefore $d^+(f)+1=d^+_T(f)\leq d^{++}_T(f)\leq d^{++}(f)+1$. Whence $f$ has the SNP in $D$.

\end{proof}

\subsection{Removing a star}
A more general statement to the following theorem is proved in \cite{gs} . Here we introduce another prove that uses the sedimentation technique of a median order.
\begin{thm}\cite{fidler}
Let $D$ be an oriented graph missing a star. Then $D$ satisfies SNC.
\end{thm}

\begin{proof}
 Orient all the missing edges of $D$ towards the center $x$ of the missing star. The obtained digraph is a tournament $T$. Let $L$ be a median order of $T$ that maximizes $\alpha$, the index of $x$ in $L$, and let $f$ denote its feed vertex. Reorient the missing edges incident to $f$ towards $f$ (if any). $L$ is also a median order of the new tournament $T'$. Note that $N^+(f)=N^+_{T'}(f)$ and we have $d^+_{T'}(f)\leq |G_L^{T'}|$. If $x\in G_L^{T'}$ and $d^+_{T'}(f)=|G_L^{T'}|$ then $sed(L)$ is a median order of $T'$ in which the index of $x$ is greater than $\alpha$, and also greater than the index of $f$. So we can give the missing edge incident to $f$ (if it exists then it is $xf$) its initial orientation (as in $T$) such that $sed(L)$ is a median order of $T$, a   contradiction to the fact that $L$ maximizes $\alpha$. So $x\notin G_L^{T'}$ or $d^+{T'}(f)<|G_L^{T'}|$. If $f=x$ then, clearly, $d^+(f)=d^+_{T'}(f)\leq |G_L^{T'}|\leq d^{++}_{T'}(f)=d^{++}(f)$. Now suppose that $f\neq x$. We have that $x$ is the only possible gained second out-neighbor vertex for $f$. If $x\notin G_L^{T'}$ then $G_L^{T'}\subseteq N^{++}(f)$, whence the result follows. If $d^+_{T'}(f)<|G_L^{T'}|$ then
$d^+(f)=d^+_{T'}(f)\leq |G_L^{T'}|-1\leq d^{++}(f)$. So $f$ has the SNP in $D$.
\end{proof}

\subsection{Removing 2 disjoint stars}

\par In this section, let $D$ be a digraph obtained from a tournament by deleting the edges of 2 disjoint stars and let $\Delta$ denote its dependency digraph. Let $S_x$ and $S_y$ be the two missing disjoint stars with centers $x$ and $y$ respectively, $A=V(S_x)\backslash x$, $B=V(S_y)\backslash y$, $K=V(S_x)\cup V(S_y)$ (the set of non whole vertices) and assume without loss of generality that $x\rightarrow y$. In \cite{gs} it is proved that if the dependency digraph of any digraph consists of isolated vertices only then it satisfies SNC. Here we consider
the case when the $\Delta$ has no isolated vertices.

\begin{thm}\label{2stars}
Let D be an oriented graph missing 2 disjoint stars. If $\Delta$ has no isolated vertex, then D satisfies SNC.
\end{thm}

\begin{proof}
Assume without loss of generality that $x\rightarrow y$. We note that the condition $\Delta$ has no isolated vertex, implies that for every $a\in A$ and $y\in B$ we have $y\rightarrow a$ and $b\rightarrow x$. We shall orient all the missing edges of $D$. First, we give every good edge a convenient orientation. For the other missing edges, let the orientation be towards the center of the 2 missing stars $S_{x}$ or $S_{y}$. The obtained digraph is a tournament $T$. Let $L$ be a median order of $T$ such that the index $k$ of $x$ is maximum and let $f$ denote its feed vertex. We know that $f$ has the SNP in $T$. We have only 5 cases:\\

 Suppose that $f$ is a whole vertex.
 In this case $N^+(f)=N^+_T(f)$. Suppose $f\rightarrow u\rightarrow v$ in $T$. Clearly $(f,u)\in D$. If $(u,v)\in D$ or is a convenient orientation then $v\in N^+(f)\cup N^{++}(f)$. Otherwise there is a missing edge $zt$ that loses to $uv$ with $t\rightarrow v$ and $u\notin N^+(f)\cup N^{++}(f)$. But $f\rightarrow u$, then $f\rightarrow t$, whence $f\rightarrow t\rightarrow v$ in D. Therefore, $N^{++}(f)=N^{++}_T(f)$ and $f$ has the SNP in $D$ as well.\\

 Suppose $f=x$.
 Orient all the edges of $S_{x}$ towards the center $x$. $L$ is a median order of the modified completion $T'$ of $D$.
 We have $N^+(f)=N^+_{T'}(f)$. Suppose $f\rightarrow u\rightarrow v$ in $T'$. If $(u,v)\in D$ or is a convenient orientation then $v\in N^+(f)\cup N^{++}(f)$. Otherwise $(u,v)=(b,y)$ for some $b\in B$, but $f=x\rightarrow y$. Thus,  $N^{++}(f)=N^{++}_{T'}(f)$ and $f$ has the SNP in $T'$ and $D$.\\

 Suppose $f=b\in B$.
 Orient the missing edge $by$ towards $b$. Again, $L$ is a median order of the modified tournament $T'$ and
 $N^+(f)=N^+_{T'}(f)$. Suppose $f\rightarrow u\rightarrow v$ in $T'$. If $(u,v)\in D$ or is a convenient orientation then $v\in N^+(f)\cup N^{++}(f)$. Otherwise $(u,v)=(b',y)$ for some $b'\in B$ or $(u,v)=(a,x)$ for some $a\in A$, however $x,y \in N^{++}(f)\cup N^{+}(f)$ because $f=b\rightarrow x\rightarrow y$ in $D$. Thus,  $N^{++}(f)=N^{++}_{T'}(f)$ and $f$ has the SNP in $T'$ and $D$.\\

 Suppose $f=y$.
 Orient the missing edges towards $y$ and let $T'$ denote the new tournament. We note that $B\subseteq N^{++}(y)\cap N^{++}_{T'}(y)$ due to the condition $\delta_{\Delta}>0$. Also, $x$ is the only possible new second neighbor of $y$ in $T'$. If $B\cup \{x\}\nsubseteqq
 G_L$ or $d^+_{T'}(y)<d^{++}_{T'}(y)$, then $d^+(y)=d^+_{T'}(y)\leq d^{++}_{T'}(y)-1\leq d^{++}(y)$. Otherwise,
 $B\cup \{x\}\nsubseteq G_L$ and $d^+_{T'}(y)=|G_L|$. In this case we consider the median order $Sed(L)$ of $T'$.
 Now the feed vertex of $sed(L)$ is different from $y$, the index of $x$ had increased, and the index of $y$ became less than the index of any vertex of $B$ which makes $Sed(L)$ a median order of $T$ also, in which the index of $x$ is greater than $k$, a contradiction.\\

 Suppose $f=a\in A$.
 Orient the missing edge $ax$ as $(x,a)$ and let $T'$ denote the new tournament. Note that $y$ is the only possible new second neighbor of $a$ in $T'$ and not in $D$. Also $x\in N^{++}_T(a)\cap N^{++}(a)$. If $d^+_{T'}(a)<d^{++}_{T'}(a)$, then $d^+(a)=d^+_{T'}(a)\leq d^{++}_{T'}(a)-1\leq d^{++}(a)$, hence $a$ has the SNP in $D$. Otherwise, $d^+_{T'}(a)=|G_L|=d^{++}_{T'}(a)$ and in particular $x\in G_L$. In this case we consider $sed(L)$ which is a median order of $T'$. Note that the feed vertex of $Sed(L)$ is different from $a$ and the index of $a$
 is less than the index of $x$ in the new order $Sed(L)$. Hence $Sed(L)$ is a median of $T$ as well, in which the index of $x$ is greater than $k$, a contradiction.\\
 So in all cases $f$ has the SNP in $D$. Therefore $D$ satisfies SNC.
 \end{proof}

\begin{thm}\label{good2starssnp}
Let D be a digraph obtained from a tournament by deleting the edges of 2 disjoint stars. If $\Delta$ has neither a source nor a sink and D has no sink, then D has at least two vertices with the SNP.
\end{thm}
\begin{proof}
$\,$\\
\textbf{claim 1:}Suppose $K=V(D)$. If $\Delta$ has no isolated vertex, then D has at least two vertices with the SNP.\\
\textbf{Proof of claim 1:}  The condition $\Delta$ has no isolated vertex implies that for every $a\in A$ and $b\in B$ we have $y\rightarrow a$ and $b\rightarrow x$. Clearly, $N^+(x)=\{y\}$, $N^{+}(y)=A$, $d^+(x)\leq 1\leq |A|\leq d^{++}(x)$, thus $x$ has the SNP.
 Let $H$ be the tournament $D-\{x,y\}$. Then $H$ has a vertex $v$ with the SNP in $H$.
 If $v\in A$, then $d^+(v)=d^+_H(v)\leq d^{++}_H(v)=d^{++}(v)$. If $v\in B$, then $d^+(v)=d^+_H(v)+1\leq d^{++}_H(v)+1=d^{++}(v)$. Whence, $v$ also has the $SNP$ in $D$.\\

 \noindent\textbf{Claim 2:} $D$ is a good digraph.\\
 \textbf{Proof of claim 2:} Let $\mathcal{I_D}$ be the interval graph of $D$. Let $C_1$ and $C_2$ be two distinct connected components of $\Delta$. Then the centers $x$ and $y$ appear in each of the these two connected components, whence $K(C_1)\cap K(C_2)\neq \phi$. Therefore, $\mathcal{I_D}$ is a connected graph, having only one connected component $\xi$. Then,
$K=K(\xi)$.\\
So if $\Delta$ is composed of non trivial strongly connected components, the result holds by lemma \ref{goodstars}.\\
Due to the condition $\Delta$ has neither a source nor a sink, $\Delta$ has a non trivial strongly connected component, hence $N^+(x)\backslash K=N^+(y)\backslash K$.
Now let $v\in K$ and assume without loss of generality that $xv$ is a missing edge. Due to the condition $\Delta$ has neither a source nor a sink, we have that $xv$ belongs to a non trivial strongly connected component of $\Delta$, and in this case $v\in R$ and  $N^+(v)\backslash K = N^+(x)\backslash K=N^+(y)\backslash K$, or $xv$ belongs to a directed path $P=xa_1,yb_1,\cdots , xa_p,yb_p$ joining 2 non trivial strongly connected components $C_1$ and $C_2$ with $xa_1\in C_1$ and $yb_p\in C_2$. There is $i>1$ such that $v=a_i$. $L=xa_{i-1},yb_{i-1},xa_i,yb_i$ is a path in $\Delta$. By the definition of losing cycles we have $N^+(x)\backslash K \subseteq N^+(b_{i-1})\backslash K \subseteq N^+(a_i)\backslash K \subseteq N^+(y)\backslash K=N^+(x)\backslash K $.
Hence $N^+(x)\backslash K = N^+(v)\backslash K$ for all $v\in K$. Since every vertex outside $K$ is adjacent to every vertex in $K$ we also have $N^-(x)\backslash K = N^-(v)\backslash K$ for all $v\in K$. This proves the second claim.\\

 Since $D$ is a good digraph, then it has a good median order \mo. If $J(x_n)=K$, then the result follows by claim 1 and theorem \ref{prince}. Otherwise, $x_n$ is whole, that is $J(x_n)=\{x_n\}$. By theorem \ref{prince}, $x_n$ has the SNP in $D$. So we need to find another vertex with the SNP in $D$. Consider the good median order $L'=x_1x_2...x_{n-1}$ of the good digraph $D'=D[\{x_1,...,x_{n-1}\}]$. Suppose first that $L'$ is stable. There is $q$ for which $Sed^q(L')=y_1...y_{n-1}$ and
$\mid N^+(y_{n-1})\backslash J(y_{n-1})\mid < \mid G_{Sed^q(L')}\backslash J(y_{n-1})\mid$ $(*)$. Note that $y_1...y_{n-1}x_n$ is also a good median order of $D$. By theorem \ref{prince} and claim 1, there is $y\in J(y_{n-1})$ that has the SNP in $D'$, more precisely $|N^+(y)|<|N^{++}(y)|$ due to $(*)$. Since $y\in J(y_{n-1})$ and $y_{n-1}\ra x_n$ then $y\ra x_n$. So $\mid N^+(y)\mid = \mid N^+_{D'}(y)\mid +1\leq \mid N^{++}(y)\mid $.\\

Now suppose that $L'$ is periodic. Since $D$ has no sink then $x_n$ has an out-neighbor $x_j$. Choose $j$ to be the greatest (so that it is the last vertex of its corresponding interval). Note that for every $q$, $x_n$ is an out-neighbor of the feed vertex of $Sed^q(L')$. So $x_j$ is not the feed vertex of any $Sed^q(L')$. Since $L'$ is periodic, $x_j$ must be a bad vertex of $Sed^q(L')$ for some integer $q$, otherwise the index of $x_j$ would
always increase during the sedimentation process. Let $q$ be such an integer and
set $Sed^q(L')=y_1...y_{n-1}$. By theorem \ref{prince} and claim 1, there is $y\in J(y_{n-1})$ that has the SNP in $D'$, more precisely $|N^+_{D'}(y)\bks J(y_{n-1})|<|G_{Sed^q(L')}\bks J(y_{n-1})|$ due to $(*)$. Since $y\in J(y_{n-1})$ and $y_{n-1}\ra x_n$ then $y\ra x_n$.
Note that $y\rightarrow x_n \rightarrow x_j$, $G_{Sed^q(L')}\cup \{x_j\}\bks J(y_{n-1})\subseteq N^{++}(y)\bks J(y_{n-1})$ and $\mid N^+_{D'}(y)\bks J(y_{n-1}) \mid =\mid G_{Sed^q(L')}\bks J(y_{n-1})\mid$.\\
Therefore $| N^+(y)| = | N^+_{D'}(y)| +1 =| N^+_{D'}(y)\bks J(y_{n-1})| +1 + |N^+_{D'}(y)\cap J(y_{n-1})|=|G_{Sed^q(L')}\bks J(y_{n-1})| +1+ |N^+_{D'}(y)\cap J(y_{n-1})\bks J(y_{n-1})|=| G_{Sed^q(L')}\cup \{x_j\}\bks J(y_{n-1})|+|N^+_{D'}(y)\cap J(y_{n-1})| \leq |N_{D}^{++}(y)\bks J(y_{n-1})|+|N^{++}_{D}(y)\cap J(y_{n-1})|\leq |N^{++}(y)|$.

\end{proof}

\subsection{Removing 3 disjoint stars}

\par In this section, $D$ is an oriented graph missing three disjoint stars $S_x$, $S_y$ and $S_z$ with centers $x$, $y$ and $z$ respectively. Set $A=V(S_x)-x$, $B=V(S_y)-x$, $C=V(S_z)-z$ and $K=A\cup B\cup C\cup\{x,y,z\}$. Let $\Delta$ denote the dependency digraph of $D$. The triangle induced by the vertices $x$, $y$ and $z$ is either a transitive triangle or a directed triangle.\\
First we will deal with the case when this triangle is directed, and assume without loss of generality that $x\rightarrow y\rightarrow z\rightarrow x$. This is a particular case of the case when the missing graph is
a disjoint union of stars such that, in the induced tournament by the centers of the missing stars, every vertex is a king.\\

\begin{thm}
Let D be an oriented graph missing 3 disjoint stars whose centers form
a directed triangle. If $\Delta$ has no isolated vertices, then D satisfies SNC.
\end{thm}

\begin{proof}
$\,$\\
\textbf{Claim:} The only possible arcs in $\Delta$ have the forms $xa\rightarrow yb$ or $yb\rightarrow zc$ or $zc\rightarrow xa$, where $a\in A$, $b\in B$ and $c\in C$.\\
\textbf{Proof of the claim:} $xa$ can not lose to $zc$ because $z\rightarrow x$ and $z\in N^{++}(x)$. Similarly $yb$ can not lose to $xa$ and $zc$ can not lose to $yb$.\\

Orient the good missing edges in a convenient way and orient the other edges toward the centers. The obtained
digraph $T$ is a tournament. Let $L$ be a median order of $T$ such that the sum of the indices of $x,y$ and $z$
is maximum. Let $f$ denote the feed vertex of $L$. Due to symmetry, we may assume that $f$ is a whole vertex or
$f=x$ or $f=a\in A$.\\

 Suppose $f$ is a whole vertex. Clearly, $N^{+}(f)=N^{+}_T(f)$. Suppose $f\rightarrow u\rightarrow v$ in $T$. If $(u,v)\in E(D)$ or $uv$ is a good missing edge then $v\in N^{+}(f)\cup N^{++}(f)$. Otherwise, there is
missing edge $rs$ that loses to $uv$ with $r\rightarrow v$ and $u\notin N^{++}(r)\cup N^{+}(r)$. But $f\rightarrow u$, then $f\rightarrow r$, whence $f\rightarrow r\rightarrow v$ and $v\in  N^{+}(f)\cup N^{++}(f)$. Thus, $N^{++}_T(f)=
N^{++}(f)$ and $f$ has the SNP in $D$.\\

Suppose $f=x$. Reorient all the missing edges incident to $x$ toward $x$. In the new tournament $T'$ we have
$N^{+}(x)=N^{+}_{T'}(x)$ and $x$ has the SNP in $T'$. Since $y\in N^{+}(x)$ and $z\in N^{++}(x)$ we have that $N^{++}(x)=N^{++}_{T'}(x)$.
Thus $x$ has the SNP in $D$.\\

Suppose that $f=a\in A$. Reorient $ax$ toward $a$. Suppose $a\rightarrow u\rightarrow v$ in the new tournament $T'$ with $v\neq y$. If $(u,v)\in E(D)$ or $uv$ is a good missing edge then $v\in N^{+}(a)\cup N^{++}(a)$. Otherwise, there is $b\in B$ and $c\in C$ such that $(u,v)=(c,z)$ and $by$ loses to $cz$, then $f\rightarrow c$ implies that $a\rightarrow y$, but $y\rightarrow z$, whence $z\in N^{++}(a)\cup N^{+}(a)$. So the only possible new second out-neighbor of $a$ is $y$, hence if $y\notin N^{++}_{T'}(a)$ then $a$ has the SNP in $D$. Suppose $y\in N^{++}_{T'}(a)$. If $d^{+}_{T'}(a)<d^{++}_{T'}(a)$ then $d^{+}(a)=d^{+}_{T'}(a)\leq d^{++}_{T'}(a)=d^{++}(a)$, hence
$a$ has the SNP in $D$. Otherwise, $d^{+}_{T'}(a)=|G_L|$ and $G_L=N^{++}_{T'}(a)$. So $x,y$ and $z$ are not bad vertices, hence the index of each increases in the median order $Sed(L)$ of $T'$. But the index of $a$ is less than
the index of $x$, then we can give $ax$ its initial orientation as in $T$ nd the same order $Sed(L)$ is a median order
of $T$. However, the sum of indices of $x,y$ and $z$ has increased. A contradiction. Thus $f$ has the SNP in $D$ and $D$ satisfies $SNC$.
\end{proof}

\begin{thm}
Let D be an oriented graph missing 3 disjoint stars whose centers form
a directed triangle. If $\Delta$ has neither a source nor a sink and D has no sink, then D has at least two vertices with the SNP.
\end{thm}

\begin{proof}
  \textbf{Claim 1:} For every $a\in A$, $b\in B$ and $c\in C$ we have:\\
$b\rightarrow x\rightarrow c\rightarrow y\rightarrow a \rightarrow z\rightarrow b$.\\
 \noindent \textbf{Proof of Claim 1:} This is due to the claim in the previous proof and the condition that $\Delta$ has neither a source nor a sink.\\

  \noindent\textbf{Claim 2:} If $K=V(D)$ then $D$ has at least 3 vertices with the SNP.\\
 \noindent \textbf{Proof of Claim 2:} Let $H=D-\{x,y,z\}$. $H$ is a tournament with no sink (dominated vertex). Then $H$ has 2 vertices $u$ and $v$ with SNP in $H$. Without loss of generality we may assume that $u\in A$. But $y\rightarrow u\rightarrow z$, the adding the vertices $x,y$ and $z$ makes $u$ gains only one vertex to its first out-neighborhood and $x$ to its second out-neighborhood. Thus, also $u$ has the SNP in $D$. Similarly, $v$ has the SNP in $D$.
Suppose, without loss of generality, that $|A|\geq|C|$. We have $C\cup\{y\}=N^{+}(x)$ and $A\cup\{z\}=N^{++}(x)$.
Hence, $d^{+}(x)=|C|+1\leq |A|+1\leq d^{++}(x)$, whence, $x$ has the SNP in $D$.\\

 \noindent\textbf{Claim 3:} $D$ is a good oriented graph.\\
 \noindent \textbf{Proof of Claim 3:} Let $\mathcal{I_D}$ be the interval graph of $D$. Let $C_1$ and $C_2$ be two distinct connected components of $\Delta$. The three centers of the missing disjoint stars must appear in each of the these two connected components, whence $K(C_1)\cap K(C_2)\neq \phi$. Therefore, $\mathcal{I_D}$ is a connected graph, having only one connected component $\xi$. Then,
$K=K(\xi)$.\\
So if $\Delta$ is composed of non trivial strongly connected components, the result holds by proposition \ref{goodstars}.\\
Due to the condition that $\Delta$ has neither a source nor a sink, $\Delta$ has a non trivial strongly connected component $C$.

Since $x$, $y$ and $z$ must appear in $C$, we have $N^{+}(x)\backslash K=N^{+}(y)\backslash K=N^{+}(z)\backslash K$. Now let $v\in K$. If $v$ appears in a non trivial strongly connected component of $\Delta$
then $N^{+}(v)\backslash K=N^{+}(x)\backslash K=N^{+}(y)\backslash K=N^{+}(z)\backslash K$.

Otherwise,due to the condition that $\Delta$ has neither a source nor a sink, $v$ appears in a directed path  $P$ of $\Delta$ joining two non trivial strongly connected components $C_1$ and $C_2$ of $\Delta$. By the definition of losing relations we can prove easily that for all $a\in K(C_1)$, $b\in K(P)$ and $c\in K(C_2)$ we have $N^+(a)\backslash K(\xi)\subseteq N^+(b)\backslash K(\xi)\subseteq N^+(c)\backslash K(\xi)$. In particular, for $a=x=c$ and $b=v$. So $N^{+}(v)\backslash K=N^{+}(x)\backslash K$. Similarly, $N^{-}(v)\backslash K=N^{-}(x)\backslash K$. This proves claim 3.\\

To conclude we apply the same argument of the proof of theorem \ref{good2starssnp}.
\end{proof}

\section{Acknowledgments}
The author thanks Pr. Amine El-Sahili for many useful discussions.

\end{document}